\documentclass[oneside,english]{amsart}
\usepackage[T1]{fontenc}
\usepackage[latin9]{inputenc}
\usepackage{url}
\usepackage{amstext}
\usepackage{amsthm}
\usepackage{amssymb}

\makeatletter
\numberwithin{equation}{section}
\numberwithin{figure}{section}
\theoremstyle{plain}
\newtheorem{thm}{\protect\theoremname}
\theoremstyle{plain}
\newtheorem{cor}[thm]{\protect\corollaryname}
\theoremstyle{definition}
\newtheorem{example}[thm]{\protect\examplename}

\makeatother

\usepackage{babel}
\providecommand{\corollaryname}{Corollary}
\providecommand{\examplename}{Example}
\providecommand{\theoremname}{Theorem}

\begin{document}

\title{Asymptotics of Certain $q$-Series}

\author{Ruiming Zhang}

\address{College of Science\\
Northwest A\&F University\\
Yangling, Shaanxi 712100\\
P. R. China.}

\email{ruimingzhang@yahoo.com}


\thanks{The work is supported by the National Natural Science Foundation
of China grants No. 11371294 and No. 11771355. }

\begin{abstract}
In this work we study complete asymptotic expansions for the q-series
$\sum_{n=1}^{\infty}\frac{1}{n^{b}}q^{n^{a}}$ and $\sum_{n=1}^{\infty}\frac{\sigma_{\alpha}(n)}{n^{b}}q^{n^{a}}$
in the scale function $(\log q)^{n}$ as $q\to1^{-}$,
where $a>0,\ q\in(0,1),\,b,\alpha\in\mathbb{C}$ and $\sigma_{\alpha}(n)$
is the divisor function $\sigma_{\alpha}(n)=\sum_{d\vert n}d^{\alpha}$.
\end{abstract}

\maketitle

\section{Preliminaries}

In this work we study complete asymptotic expansions for the q-series
$\sum_{n=1}^{\infty}\frac{1}{n^{b}}q^{n^{a}}$ and $\sum_{n=1}^{\infty}\frac{\sigma_{\alpha}(n)}{n^{b}}q^{n^{a}}$
in the scale function $(\log q)^{n}$ as $q\to1^{-}$, where $a>0,\ q\in(0,1),\,b,\alpha\in\mathbb{C}$
and $\sigma_{\alpha}(n)$ is the divisor function $\sigma_{\alpha}(n)=\sum_{d\vert n}d^{\alpha}$.
Unlike methods used \cite{Berndt,Bringmann}, our method does not
apply Fourier transform or the modular properties, it can not give
$\sum_{n=1}^{\infty}\frac{1}{n^{b}}q^{n^{a}}$ a complete asymptotic
expansion in exponential scales when $a=2$ and $b$ is an even integer.
However, this shortcoming can be overcome by applying the functional
equations for the corresponding zeta functions which are equivalent
to the symmetry $x\to1/x$. 

The Euler gamma function is defined by

\begin{equation}
\Gamma\left(z\right)=\int_{0}^{\infty}e^{-x}x^{z-1}dx,\quad\Re(z)>0,\label{eq:1.1}
\end{equation}
and its analytic continuation is given by

\begin{equation}
\Gamma\left(z\right)=\int_{1}^{\infty}e^{-x}x^{z-1}dx+\sum_{n=0}^{\infty}\frac{\left(-1\right)^{n}}{n!\left(n+z\right)},\quad z\in\mathbb{C}\backslash\mathbb{N}_{0}.\label{eq:1.2}
\end{equation}
 Let $a,b\in\mathbb{R}$ and $a<b$, it is known that \cite{Andrews,DLMF,Rademacher}
\begin{equation}
\Gamma(\sigma+it)=\mathcal{O}\left(e^{-\pi\left|t\right|/2}\left|t\right|^{\sigma-1/2}\right),\quad t\in\mathbb{R}\label{eq:1.3}
\end{equation}
as $t\to\pm\infty$, uniformly with respect to $\sigma\in[a,b]$.
The digamma function is defined by 
\begin{equation}
\psi(z)=\frac{\Gamma^{'}(z)}{\Gamma(z)},\quad z\in\mathbb{C}\label{eq:1.4}
\end{equation}
and the Euler's constant is
\begin{equation}
\gamma=-\psi(1)\approx0.577216.\label{eq:1.5}
\end{equation}

The Riemann zeta function $\zeta\left(s\right)$ is defined by
\begin{equation}
\zeta(s)=\sum_{n=1}^{\infty}\frac{1}{n^{s}},\quad\Re(s)>1,\label{eq:1.6}
\end{equation}
then its analytic continuation, which is also denoted as $\zeta(s)$,
is an meromorphic function that has a simple pole at $1$ with residue
$1$. The meromorphic function $\zeta(s)$ satisfies the functional
equation \cite{Andrews,Apostol,DLMF,Rademacher}
\begin{equation}
\zeta(s)=2^{s}\pi^{s-1}\Gamma(1-s)\sin\left(\frac{\pi s}{2}\right)\zeta(1-s).\label{eq:1.7}
\end{equation}
For $\alpha,\beta\in\mathbb{R}$ and $\alpha\le\sigma\le\beta$, it
is known that \cite{Rademacher} 
\begin{equation}
\zeta(\sigma+it)=\mathcal{O}\left(\left|t\right|^{\left|\alpha\right|+1/2}\right)\label{eq:1.8}
\end{equation}
as $t\to\pm\infty$, uniformly with respect to $\sigma$. The Stieltjes
constants $\gamma_{n}$ are the coefficients in the Laurent expansion,
\begin{equation}
{\displaystyle \zeta(s)=\frac{1}{s-1}+\sum_{n=0}^{\infty}\frac{(-1)^{n}}{n!}\gamma_{n}(s-1)^{n},}\label{eq:1.9}
\end{equation}
where $\gamma_{0}=\gamma$ and $\gamma_{1}\approx-0.0728158$. Moreover,
the Glaisher's constant $A\approx1.28243$ is defined as 
\begin{equation}
\log A=\frac{1}{12}-\zeta^{'}(-1).\label{eq:1.10}
\end{equation}

The Bernoulli numbers $B_{n}$ are defined by

\begin{equation}
\frac{z}{e^{z}-1}=\sum_{n=0}^{\infty}\frac{B_{n}}{n!}z^{n},\quad\left|z\right|<2\pi.\label{eq:1.11}
\end{equation}
Then
\begin{equation}
B_{0}=1,\ B_{1}=-\frac{1}{2},\ B_{2n-1}=0,\ n\in\mathbb{N}.\label{eq:1.12}
\end{equation}
 By (\ref{eq:1.7}) we get
\begin{equation}
\zeta(-2n)=0,\ \zeta(1-2n)=-\frac{B_{2n}}{2n},\quad n\in\mathbb{N}.\label{eq:1.13}
\end{equation}

The function $\sigma_{\alpha}(n)$ for $\alpha\in\mathbb{C}$ is defined
as the sum of the $\alpha$-th powers of the positive divisors of
$n$, \cite{Apostol,DLMF}
\begin{equation}
{\displaystyle \sigma_{\alpha}(n)=\sum_{d\mid n}d^{\alpha},}\label{eq:1.14}
\end{equation}
where $d\vert n$ stands for \textquotedbl$d${} divides $n$\textquotedbl .
We also use the notations $d(n)=\sigma_{0}$ and $\sigma(n)=\sigma_{1}(n).$
It is known that \cite{Apostol,DLMF}
\begin{equation}
\sum_{n=1}^{\infty}\frac{\sigma_{\alpha}(n)}{n^{s}}=\zeta(s)\zeta(s-\alpha),\quad\Re(s)>\max\left\{ 1,\Re(\alpha)+1\right\} \label{eq:1.15}
\end{equation}
 and
\begin{equation}
\sum_{n=1}^{\infty}\frac{d(n)}{n^{s}}=\zeta^{2}(s)\quad\Re(s)>1.\label{eq:1.16}
\end{equation}

\section{Main Results}
\begin{thm}
\label{thm:1} Given a positive integer $k$, let $a_{j}\in\mathbb{N},\ b_{j}\in\mathbb{C}$
for all $j$ satisfying $1\le j\le k$.

If 
\begin{equation}
\prod_{j=1}^{k}\zeta\left(a_{j}s+b_{j}\right)=\sum_{n=1}^{\infty}\frac{f_{k}(n)}{n^{s}},\quad\Re(s)>\max_{1\le j\le k}\left\{ \frac{1-\Re(b_{j})}{a_{j}}\right\} ,\label{eq:2.1}
\end{equation}
where $\zeta(s)$ is the Riemann zeta function, then for all $x,a>0$,
$b\in\mathbb{C}$ and $c>0$ satisfying $c>\max_{1\le j\le k}\left\{ \frac{1-\Re(b_{j}+b)}{aa_{j}}\right\} $
we have
\begin{equation}
\frac{1}{2\pi i}\int_{c-i\infty}^{c+i\infty}\Gamma(s)\prod_{j=1}^{k}\zeta\left(aa_{j}s+b_{j}+b\right)\frac{ds}{x^{s}}=\sum_{n=1}^{\infty}\frac{f_{k}(n)}{n^{b}}e^{-n^{a}x}.\label{eq:2.2}
\end{equation}
 Furthermore,
\begin{align}
\sum_{n=1}^{\infty}\frac{f_{k}(n)}{n^{b}}e^{-n^{a}x} & =\sum_{j}\text{Residue}\left\{ g(s),s=\frac{1-b-b_{j}}{aa_{j}}\right\} +\sum_{n}\text{Residue}\left\{ g(s),s=-n\right\} \label{eq:2.3}
\end{align}
as $x\to0^{+}$, where the first sum is over all the distinct pairs
$a_{j}\,b_{j}\ j=1,\dots,k$ while the last sum is over all nonnegative
integers $n$ such that
\begin{equation}
-n\neq\frac{1-b-b_{j}}{aa_{j}},\quad j=1,\dots,k.\label{eq:2.4}
\end{equation}
\end{thm}

\begin{proof}
For $\Re(s)>\max_{1\le j\le k}\left\{ \frac{1-\Re(b_{j}+b)}{aa_{j}}\right\} $,
since each factor of $\prod_{j=1}^{k}\zeta\left(aa_{j}s+b_{j}+b\right)$
is an absolute convergent Dirichlet series, then the product itself
is also an absolute convergent Dirichlet series. Let $s_{0}$ be any
complex number satisfying
\[
\sigma_{0}=\Re(s_{0})>\max_{1\le j\le k}\left\{ \frac{1-\Re(b_{j}+b)}{aa_{j}}\right\} ,
\]
then by the theory of Dirichlet series we know the partial sums $\sum_{n\le x}\frac{f_{k}(n)}{n^{as_{0}+b}}$
are absolutely and uniformly bounded for all $x>1$. Let $N$ be a
large positive integer and
\[
M=\sum_{n=1}^{\infty}\frac{\left|f_{k}(n)\right|}{n^{a\sigma_{0}+\Re(b)}},\ s=s_{0}+1,\ a=N-1,\ b=N
\]
 in Lemma 2 of section 11.6 in \cite{Apostol} to get $\left|f_{k}(N)\right|\le4MN^{\sigma_{0}}.$
Hence,
\begin{equation}
\sum_{n=1}^{\infty}\frac{\left|f_{k}(N)\right|}{n^{\Re(b)}}e^{-n^{a}x}<\infty,\quad a,x>0,\ b\in\mathbb{C}.\label{eq:2.5}
\end{equation}
By the inverse Mellin transform of (\ref{eq:1.1}) we get
\begin{equation}
\frac{1}{2\pi i}\int_{c-i\infty}^{c+i\infty}\Gamma(s)\frac{ds}{x^{s}}=e^{-x}\label{eq:2.6}
\end{equation}
for all $x,c>0$.

Let 
\[
g(s)=\Gamma(s)\prod_{j=1}^{k}\zeta\left(aa_{j}s+b_{j}+b\right)x^{-s},
\]
then for any positive $c$ satisfying $c>\max_{1\le j\le k}\left\{ \frac{1-\Re(b_{j}+b)}{aa_{j}}\right\} $,
by (\ref{eq:2.6}) we get
\begin{align}
 & \frac{1}{2\pi i}\int_{c-i\infty}^{c+i\infty}g(s)ds=\frac{1}{2\pi i}\int_{c-i\infty}^{c+i\infty}\Gamma(s)\left(\sum_{n=1}^{\infty}\frac{f_{k}(n)}{n^{as+b}}\right)\frac{ds}{x^{s}}\label{eq:2.7}\\
 & =\sum_{n=1}^{\infty}\frac{f_{k}(n)}{n^{b}}\frac{1}{2\pi i}\int_{c-i\infty}^{c+i\infty}\Gamma(s)\frac{ds}{\left(n^{a}x\right)^{s}}=\sum_{n=1}^{\infty}\frac{f_{k}(n)}{n^{b}}e^{-n^{a}x},\nonumber 
\end{align}
where we have applied (\ref{eq:2.5}) and the Fubini's theorem to
exchange the order of summation and integration. 

Since $\zeta(s)$ has a simple pole at $s=1$ and $\Gamma(s)$ has
simple poles at all non-positive integers, then all the possible poles
of the meromorphic function $g(s)$ are
\[
s=\frac{1-b-b_{j}}{aa_{j}},\quad j=1,\dots,k
\]
 and all non-positive integers. Let $N\in\mathbb{N}$ and $M\in\mathbb{R}$
such that 
\[
N>\max_{1\le j\le k}\left\{ \frac{1+|b+b_{j}|}{aa_{j}}\right\} +1,\quad M>\max_{1\le j\le k}\left\{ \frac{1+|b+b_{j}|}{aa_{j}}\right\} ,
\]
we integrate $g(s)$ over the rectangular contour $\mathcal{R}(M,N)$
with vertices,
\[
c-iM,\ c+iM,\ -N-\frac{1}{2}+iM,\ -N-\frac{1}{2}-iM.
\]
 Then by Cauchy's theorem we have
\begin{equation}
\int_{\mathcal{R}(M,N)}\frac{g(s)ds}{2\pi i}=\sum_{j}\text{Residue}\left\{ g(s),s=\frac{1-b-b_{j}}{aa_{j}}\right\} +\sum_{n}\text{Residue}\left\{ g(s),s=-n\right\} ,\label{eq:2.8}
\end{equation}
where the first sum is over all the distinct pairs from $a_{j}\,b_{j},\ j=1,\dots,k$
whereas the last sum is over all $n$ satisfying $0\le n\le N$ and
(\ref{eq:2.4}).

On the other hand, we also have
\begin{equation}
\int_{\mathcal{R}(M,N)}\frac{g(s)ds}{2\pi i}=\left\{ \int_{c-iM}^{c+iM}-\int_{-\frac{2N+1}{2}-iM}^{-\frac{2N+1}{2}+iM}\right\} \frac{g(s)ds}{2\pi i}+\left\{ \int_{-\frac{2N+1}{2}-iM}^{c-iM}-\int_{-\frac{2N+1}{2}+iM}^{c+iM}\right\} \frac{g(s)ds}{2\pi i}\label{eq:2.9}
\end{equation}
Fix $N$ and $x$, by (\ref{eq:1.3}) and (\ref{eq:1.8}), since the
integrands of the last two integrals have the estimate
\[
g(s)=\mathcal{O}\left(e^{-\left(\pi/2-\epsilon\right)M}\right),\quad M\to\infty,
\]
where $\epsilon$ is an arbitrary positive number such that $0<\epsilon<\frac{\pi}{2}$,
then the last two integrals have limit $0$ as $M\to\infty$. Then
by taking limit $M\to\infty$ in (\ref{eq:2.9}) and (\ref{eq:2.8})
we get
\begin{align}
 & \frac{1}{2\pi i}\int_{c-i\infty}^{c+i\infty}g(s)ds=\frac{1}{2\pi i}\int_{-\frac{2N+1}{2}-i\infty}^{-\frac{2N+1}{2}+i\infty}g(s)ds\label{eq:2.10}\\
 & +\sum_{j}\text{Residue}\left\{ g(s),s=\frac{1-b-b_{j}}{aa_{j}}\right\} +\sum_{n}\text{Residue}\left\{ g(s),s=-n\right\} ,\nonumber 
\end{align}
where the summations are the same as in (\ref{eq:2.8}). 

Since
\[
\left|\frac{1}{2\pi i}\int_{-\frac{2N+1}{2}-i\infty}^{-\frac{2N+1}{2}+i\infty}g(s)ds\right|\le\frac{x^{N+1/2}}{2\pi}\int_{-\frac{2N+1}{2}-i\infty}^{-\frac{2N+1}{2}+i\infty}\left|\Gamma(s)\prod_{j=1}^{k}\zeta\left(aa_{j}s+b_{j}+b\right)\right|dt,
\]
then again by (\ref{eq:1.3}) and (\ref{eq:1.8}) we get
\begin{equation}
\frac{1}{2\pi i}\int_{-\frac{2N+1}{2}-i\infty}^{-\frac{2N+1}{2}+i\infty}g(s)ds=o\left(x^{N}\right)\label{eq:2.11}
\end{equation}
as $x\to0$. Then by (\ref{eq:2.10}) and (\ref{eq:2.11}) we get
\begin{equation}
\frac{1}{2\pi i}\int_{c-i\infty}^{c+i\infty}g(s)ds=\sum_{j}\text{Residue}\left\{ g(s),s=\frac{1-b-b_{j}}{aa_{j}}\right\} +\sum_{n}\text{Residue}\left\{ g(s),s=-n\right\} \label{eq:2.12}
\end{equation}
as $x\to0$, where the first sum is over all the distinct pairs from
$a_{j}\,b_{j},\ j=1,\dots,k$ while the last sum is over all nonnegative
integers $n$ satisfying (\ref{eq:2.4}). Finally, (\ref{eq:2.2})
is obtained by combining (\ref{eq:2.7}) and (\ref{eq:2.12}).
\end{proof}
\begin{cor}
\label{cor:riemann-theta} Let $a,x>0,\ b\in\mathbb{C}$. If $b\neq1+an,\quad n\in\mathbb{N}\cup\left\{ 0\right\} $,
then
\begin{equation}
\sum_{n=1}^{\infty}\frac{e^{-n^{a}x}}{n^{b}}=\frac{x^{\frac{b-1}{a}}}{a}\Gamma\left(\frac{1-b}{a}\right)+\sum_{n=0}^{\infty}\frac{\left(-x\right)^{n}}{n!}\zeta(b-an)\label{eq:2.13}
\end{equation}
as $x\to0$. 

If there exists a $n_{0}\in\mathbb{N}\cup\left\{ 0\right\} $ such
that $b=1+an_{0}$, then 
\begin{align}
\sum_{n=1}^{\infty}\frac{e^{-n^{a}x}}{n^{b}} & =\frac{(-x)^{n_{0}}(\gamma a+\psi(n_{0}+1)-\log(x))}{an_{0}!}+\sum_{\begin{array}{c}
n=0\\
n\neq n_{0}
\end{array},}^{\infty}\frac{(-x)^{n}}{n!}\zeta(b-an)\label{eq:2.14}
\end{align}
as $x\to0$\textup{.}
\end{cor}

\begin{proof}
When $b\neq1+an,\quad n\in\mathbb{N}\cup\left\{ 0\right\} ,$ the
integrand $\frac{\Gamma(s)\zeta(as+b)}{x^{s}}$ is meromorphic and
has the following simple poles
\[
s=\frac{1-b}{a},0,-1,-2,\dots
\]
 with residues 
\begin{align*}
\text{Residue}\left\{ \frac{\Gamma(s)\zeta(as+b)}{x^{s}},s=\frac{1-b}{a}\right\}  & =\frac{x^{(b-1)/a}}{a}\Gamma\left(\frac{1-b}{a}\right)\\
\text{Residue}\left\{ \frac{\Gamma(s)\zeta(as+b)}{x^{s}},s=-n\right\}  & =\frac{(-x)^{n}}{n!}\zeta(b-an).
\end{align*}
 Then (\ref{eq:2.13}) is obtained by applying Theorem \ref{thm:1}.

When $b=1+an_{0}$ for some nonnegative integer $n_{0}$, then 
\[
\frac{\Gamma(s)\zeta(as+b)}{x^{s}}=\frac{\Gamma(s)\zeta\left(a(s+n_{0})+1\right)}{x^{s}}
\]
 has a double pole at $-n_{0}$ with residue
\[
\mbox{Residue}\left\{ \frac{\Gamma(s)\zeta(as+b)}{x^{s}},s=-n_{0}\right\} =\frac{(-1)^{n_{0}}x^{n_{0}}(\gamma a+\psi(n_{0}+1)-\log(x))}{an_{0}!},
\]
all the other nonpositive integers are simple poles with residues,
\[
\mbox{Residue}\left\{ \frac{\Gamma(s)\zeta(as+b)}{x^{s}},s=-n\neq-n_{0}\right\} =\frac{(-x)^{n}}{n!}\zeta(b-an).
\]
Then by Theorem \ref{thm:1} we have
\begin{align*}
\sum_{n=1}^{\infty}\frac{e^{-n^{a}x}}{n^{b}} & =\frac{(-x)^{n_{0}}(\gamma a+\psi(n_{0}+1)-\log(x))}{an_{0}!}+\sum_{\begin{array}{c}
n=0\\
an\neq(b-1)
\end{array},}^{\infty}\frac{(-x)^{n}}{n!}\zeta(b-an)
\end{align*}
as $x\to0$. 
\end{proof}
\begin{example}
When $a=2,\ b=0$ we have
\begin{equation}
\sum_{n=1}^{\infty}e^{-n^{2}x}=\frac{1}{2\sqrt{\pi x}}+\sum_{n=0}^{\infty}\frac{\left(-x\right)^{n}}{n!}\zeta(-2n)=\frac{1}{2\sqrt{\pi x}}\label{eq:2.16}
\end{equation}
 as $x\to0$, which means the error term is better than any $x^{n}$.
When $a=2,\ b=1$, the double pole happens at $n_{0}=0$, then 
\begin{equation}
\sum_{n=1}^{\infty}\frac{e^{-n^{2}x}}{n}=\frac{\gamma-\log(x)}{2}-\sum_{n=1}^{\infty}\frac{B_{2n}(-x)^{n}}{n!(2n)}\label{eq:2.17}
\end{equation}
 as $x\to0$. When $a=2,\ b=-1$, then $2n+1=-1$ has no nonnegative
integer solutions. Thus, 
\begin{equation}
\sum_{n=1}^{\infty}ne^{-n^{2}x}=\frac{1}{2x}+\frac{1}{2x}\sum_{n=1}^{\infty}\frac{B_{2n}\left(-x\right)^{n}}{n!},\label{eq:2.18}
\end{equation}
or
\begin{equation}
\sum_{n=-\infty}^{\infty}|n|e^{-n^{2}x}=\frac{1}{x}+\frac{1}{x}\sum_{n=1}^{\infty}\frac{B_{2n}\left(-x\right)^{n}}{n!}\label{eq:2.19}
\end{equation}
as $x\to0$. 
\end{example}

\begin{cor}
\label{cor:d-n}For all $x,a>0$, $b\in\mathbb{C}$ and $c>\frac{1-\Re(b)}{a}$
we have
\begin{equation}
\frac{1}{2\pi i}\int_{c-i\infty}^{c+i\infty}\Gamma(s)\zeta^{2}(as+b)\frac{ds}{x^{s}}=\sum_{n=1}^{\infty}\frac{d(n)}{n^{b}}e^{-n^{a}x}.\label{eq:2.20}
\end{equation}
 Furthermore, if $an\neq b-1$ for all nonnegative integers $n$,
then 
\begin{equation}
\sum_{n=1}^{\infty}\frac{d(n)}{n^{b}}e^{-n^{a}x}=\frac{x^{\frac{b-1}{a}}\Gamma\left(\frac{1-b}{a}\right)\left(\psi\left(\frac{1-b}{a}\right)+2\gamma a-\log(x)\right)}{a^{2}}+\sum_{n=0}^{\infty}\frac{(-x)^{n}}{n!}\zeta^{2}\left(b-an\right)\label{eq:2.21}
\end{equation}
 as $x\to0^{+}$\textup{. }

If $am=b-1$ for certain nonnegative $m$, then 
\begin{align}
 & \sum_{n=1}^{\infty}\frac{d(n)}{n^{b}}e^{-n^{a}x}=\frac{(-x)^{m}}{a^{2}m!}\left\{ (a\gamma)^{2}-2a^{2}\gamma_{1}+\left(2a\gamma-\log x\right)\psi(m+1)-2a\gamma\log x\right.\label{eq:2.22}\\
 & +\left.\frac{\psi^{2}(m+1)-\psi^{(1)}(m+1)+\log^{2}(x)}{2}+\frac{\pi^{2}}{6}\right\} +\sum_{\begin{array}{c}
n=0\\
n\neq m
\end{array}}^{\infty}\frac{(-x)^{n}}{n!}\zeta^{2}\left(b-an\right),\nonumber 
\end{align}
as $x\to0^{+}$\textup{.}
\end{cor}

\begin{proof}
When $an\neq b-1$ for all nonnegative integers $n$, then the meromorphic
function $\Gamma(s)\zeta^{2}(as+b)x^{-s}$ has a double pole at $s=(1-b)/a$
with residue
\[
\mbox{Residue}\left\{ \frac{\Gamma(s)\zeta^{2}(as+b)}{x^{s}},s=\frac{1-b}{a}\right\} =\frac{x^{\frac{b-1}{a}}\Gamma\left(\frac{1-b}{a}\right)\left(\psi\left(\frac{1-b}{a}\right)+2\gamma a-\log(x)\right)}{a^{2}}
\]
and simple poles at all nonpositive integers $n\in\mathbb{N}_{0}$
with residue
\[
\mbox{Residue}\left\{ \frac{\Gamma(s)\zeta^{2}(as+b)}{x^{s}},s=-n\right\} =\frac{(-x)^{n}}{n!}\zeta^{2}\left(b-an\right).
\]
Then by Theorem \ref{thm:1} we get
\[
\sum_{n=1}^{\infty}\frac{d(n)}{n^{b}}e^{-n^{a}x}=\frac{x^{\frac{b-1}{a}}\Gamma\left(\frac{1-b}{a}\right)\left(\psi\left(\frac{1-b}{a}\right)+2\gamma a-\log(x)\right)}{a^{2}}+\sum_{n=0}^{\infty}\frac{(-x)^{n}}{n!}\zeta^{2}\left(b-an\right)
\]
as $x\to0$.

When $\alpha=0,\ am=b-1$ for certain nonnegative integer $m$, then
the meromorphic function $\Gamma(s)\zeta^{2}(as+b)x^{-s}$ has a triple
pole at $s=-m$ with residue
\begin{align*}
 & \mbox{Residue}\left\{ \frac{\Gamma(s)\zeta^{2}(as+b)}{x^{s}},s=-m\right\} \\
 & =\frac{(-x)^{m}}{a^{2}m!}\left\{ (a\gamma)^{2}-2a^{2}\gamma_{1}+\left(2a\gamma-\log x\right)\psi(m+1)-2a\gamma\log x\right.\\
 & +\left.\frac{\psi^{2}(m+1)-\psi^{(1)}(m+1)+\log^{2}(x)}{2}+\frac{\pi^{2}}{6}\right\} .
\end{align*}
 It has simple poles at all other nonpositive integers with residue
\[
\mbox{Residue}\left\{ \frac{\Gamma(s)\zeta^{2}(as+b)}{x^{s}},s=-n\right\} =\frac{(-x)^{n}}{n!}\zeta^{2}\left(b-an\right).
\]
Then by Theorem \ref{thm:1} we get
\begin{align*}
 & \sum_{n=1}^{\infty}\frac{d(n)}{n^{b}}e^{-n^{a}x}=\frac{(-x)^{m}}{a^{2}m!}\left\{ (a\gamma)^{2}-2a^{2}\gamma_{1}+\left(2a\gamma-\log x\right)\psi(m+1)-2a\gamma\log x\right.\\
 & +\left.\frac{\psi^{2}(m+1)-\psi^{(1)}(m+1)+\log^{2}(x)}{2}+\frac{\pi^{2}}{6}\right\} +\sum_{\begin{array}{c}
n=0\\
n\neq m
\end{array}}^{\infty}\frac{(-x)^{n}}{n!}\zeta^{2}\left(b-an\right)
\end{align*}
as $x\to0^{+}$.
\end{proof}
\begin{example}
Let $a=2,\ b=2$, then by (\ref{eq:2.21}) to get
\begin{equation}
\sum_{n=1}^{\infty}\frac{d(n)}{n^{2}}e^{-n^{2}x}=\frac{\sqrt{x\pi}\left(\log x-\psi(-\frac{1}{2})-4\gamma\right)}{2}+\frac{\pi^{4}}{36}
\end{equation}
 as $x\to0^{+}$, the remainder here is better than any $x^{n}$.
Let $a=2,\ b=1$, then by (\ref{eq:2.22}) to get
\begin{align}
\sum_{n=1}^{\infty}\frac{d(n)}{n}e^{-n^{2}x} & =\frac{6\log^{2}x-45\log x+6\gamma^{2}+\pi^{2}-24\gamma_{1}}{12}+\frac{1}{4}\sum_{n=1}^{\infty}\frac{B_{2n}^{2}\left(-x\right)^{n}}{n(n+1)!}
\end{align}
 as $x\to0^{+}$.
\end{example}

\begin{cor}
\label{cor:sigma-n}Let $\alpha\in\mathbb{C}$ and $\alpha\neq0$,
then for all $x,a>0$, $b\in\mathbb{C}$ and $c>\frac{\max_{1\le j\le k}\left\{ 1-\Re(b),1-\Re(b-\alpha)\right\} }{a}$
we have
\begin{equation}
\frac{1}{2\pi i}\int_{c-i\infty}^{c+i\infty}\Gamma(s)\zeta(as+b)\zeta(as+b-\alpha)\frac{ds}{x^{s}}=\sum_{n=1}^{\infty}\frac{\sigma_{\alpha}(n)}{n^{b}}e^{-n^{a}x}.\label{eq:2.23}
\end{equation}
 Furthermore, if $an\neq b-1$ and $an\neq b-1-\alpha$ for all nonnegative
integers $n\in\mathbb{N}_{0}$, then
\begin{align}
 & \sum_{n=1}^{\infty}\frac{\sigma_{\alpha}(n)}{n^{b}}e^{-n^{a}x}=\sum_{n=0}^{\infty}\frac{(-x)^{n}}{n!}\zeta(b-an)\zeta(b-an-\alpha)\label{eq:2.24}\\
 & +\frac{x^{(b-1)/a}}{a}\Gamma\left(\frac{1-b}{a}\right)\zeta\left(1-\alpha\right)+\frac{x^{(b-1-\alpha)/a}}{a}\Gamma\left(\frac{1-b+\alpha}{a}\right)\zeta\left(1+\alpha\right)\nonumber 
\end{align}
 as $x\to0^{+}$. 

If $am=b-1$ for certain nonnegative integer $m$ and $an\neq b-1-\alpha$
for all nonnegative integers $n\in\mathbb{N}_{0}$, then
\begin{align}
 & \sum_{n=1}^{\infty}\frac{\sigma_{\alpha}(n)}{n^{b}}e^{-n^{a}x}=\frac{(-x)^{m}\left(a\zeta'(1-\text{\ensuremath{\alpha}})+\zeta(1-\text{\ensuremath{\alpha}})(\gamma a+\psi(m+1)-\log(x))\right)}{am!}\label{eq:2.25}\\
 & +\frac{x^{m-\alpha/a}}{a}\Gamma\left(\frac{\alpha-ma}{a}\right)\zeta\left(1+\alpha\right)+\sum_{\begin{array}{c}
n=0\\
n\neq m
\end{array}}^{\infty}\frac{(-x)^{n}}{n!}\zeta(1-a(n-m))\zeta(1-\alpha-a(n-m))\nonumber 
\end{align}
as $x\to0^{+}$. 

If $an\neq b-1$ for all nonnegative integers $n\in\mathbb{N}_{0}$
and $am=b-1-\alpha$ for certain $m\in\mathbb{N}_{0}$, then 
\begin{align}
 & \sum_{n=1}^{\infty}\frac{\sigma_{\alpha}(n)}{n^{b}}e^{-n^{a}x}=\frac{x^{m+\alpha/a}}{a}\Gamma\left(-\frac{am+\alpha}{a}\right)\zeta\left(1-\alpha\right)\label{eq:2.26}\\
 & +\frac{(-x)^{m}\left(a\zeta'(1+\alpha)+\zeta(1+\alpha)(\gamma a+\psi(m+1)-\log(x))\right)}{am!}\nonumber \\
 & +\sum_{\begin{array}{c}
n=0\\
n\neq m
\end{array}}^{\infty}\frac{(-x)^{n}}{n!}\zeta(1+\alpha-a(n-m))\zeta(1-a(n-m))\nonumber 
\end{align}
 as $x\to0^{+}$. 

If $am=b-1$ and $\alpha=a(m-\ell)$ for certain nonnegative integers
$m,\ell\in\mathbb{N}_{0}$ with $m\neq\ell$, then
\begin{align}
 & \sum_{n=1}^{\infty}\frac{\sigma_{\alpha}(n)}{n^{b}}e^{-n^{a}x}=\sum_{\begin{array}{c}
n=0\\
n\neq m,\ell
\end{array}}^{\infty}\frac{(-x)^{n}}{n!}\zeta(1-a(n-m))\zeta(1-a(n-\ell))\label{eq:2.27}\\
 & +\frac{(-x)^{m}\left(a\zeta'(1-a(m-\ell))+\zeta(1-a(m-\ell))(\gamma a+\psi(m+1)-\log(x))\right)}{am!}\nonumber \\
 & +\frac{(-x)^{\ell}\left(a\zeta'(1+a(m-\ell))+\zeta(1+a(m-\ell))(\gamma a+\psi(\ell+1)-\log(x))\right)}{a\ell!}\nonumber 
\end{align}
 as $x\to0^{+}$.
\end{cor}

\begin{proof}
When $\alpha\neq0$, $an\neq b-1$ and $an\neq b-1-\alpha$ for all
nonnegative integers $n\in\mathbb{N}_{0}$, the meromorphic function
$\Gamma(s)\zeta(as+b)\zeta(as+b-\alpha)x^{-s}$ has simple poles at
\[
\frac{1-b}{a},\ \frac{1-b+\alpha}{a},0,-1,-2,\dots
\]
 with residues 
\[
\text{Residue}\left\{ \frac{\Gamma(s)\zeta(as+b)\zeta(as+b-\alpha)}{x^{s}},s_{1}=\frac{1-b}{a}\right\} =\frac{x^{(b-1)/a}}{a}\Gamma\left(\frac{1-b}{a}\right)\zeta\left(1-\alpha\right),
\]
 
\[
\text{Residue}\left\{ \frac{\Gamma(s)\zeta(as+b)\zeta(as+b-\alpha)}{x^{s}},s_{2}=\frac{1-b+\alpha}{a}\right\} =\frac{x^{(b-1-\alpha)/a}}{a}\Gamma\left(\frac{1-b+\alpha}{a}\right)\zeta\left(1+\alpha\right)
\]
and
\[
\text{Residue}\left\{ \frac{\Gamma(s)\zeta(as+b)\zeta(as+b-\alpha)}{x^{s}},s_{3}=-n\right\} =\frac{(-x)^{n}}{n!}\zeta(b-an)\zeta(b-an-\alpha)
\]
 for $n\in\mathbb{N}_{0}$. Then by Theorem \ref{thm:1} we get
\begin{align*}
 & \sum_{n=1}^{\infty}\frac{\sigma_{\alpha}(n)}{n^{b}}e^{-n^{a}x}=\sum_{n=0}^{\infty}\frac{(-x)^{n}}{n!}\zeta(b-an)\zeta(b-an-\alpha)\\
 & +\frac{x^{(b-1)/a}}{a}\Gamma\left(\frac{1-b}{a}\right)\zeta\left(1-\alpha\right)+\frac{x^{(b-1-\alpha)/a}}{a}\Gamma\left(\frac{1-b+\alpha}{a}\right)\zeta\left(1+\alpha\right)
\end{align*}
 as $x\to0^{+}$.

When $am=b-1$ for certain nonnegative integer $m$ and $an\neq b-1-\alpha$
for all nonnegative integers $n\in\mathbb{N}_{0}$, then the meromorphic
function $\Gamma(s)\zeta(as+b)\zeta(as+b-\alpha)x^{-s}$ has a double
pole at $s=-m$ with residue
\begin{align*}
 & \text{Residue}\left\{ \frac{\Gamma(s)\zeta(as+b)\zeta(as+b-\alpha)}{x^{s}},s=-m\right\} \\
 & =\frac{(-x)^{m}\left(a\zeta'(1-\alpha)+\zeta(1-\alpha)(\gamma a+\psi(m+1)-\log(x))\right)}{am!},
\end{align*}
 and a simple pole at $s=-m+\frac{\alpha}{a}$ with residue
\[
\text{Residue}\left\{ \frac{\Gamma(s)\zeta(as+b)\zeta(as+b-\alpha)}{x^{s}},s=-m+\frac{\alpha}{a}\right\} =\frac{x^{m-\alpha/a}}{a}\Gamma\left(\frac{\alpha-ma}{a}\right)\zeta\left(1+\alpha\right)
\]
and simple poles at all nonpositive integers other than $-m$ with
residues
\[
\text{Residue}\left\{ \frac{\Gamma(s)\zeta(as+b)\zeta(as+b-\alpha)}{x^{s}},s=-n\right\} =\frac{(-x)^{n}}{n!}\zeta(b-an)\zeta(b-\alpha-an).
\]
Hence,
\begin{align*}
 & \sum_{n=1}^{\infty}\frac{\sigma_{\alpha}(n)}{n^{b}}e^{-n^{a}x}=\frac{(-x)^{m}\left(a\zeta'(1-\alpha)+\zeta(1-\alpha)(\gamma a+\psi(m+1)-\log(x))\right)}{am!}\\
 & +\frac{x^{m-\alpha/a}}{a}\Gamma\left(\frac{\alpha-ma}{a}\right)\zeta\left(1+\alpha\right)+\sum_{\begin{array}{c}
n=0\\
n\neq m
\end{array}}^{\infty}\frac{(-x)^{n}}{n!}\zeta(b-an))\zeta(b-\alpha-an)
\end{align*}
as $x\to0^{+}$. 

When $an\neq b-1$ for all nonnegative integers $n\in\mathbb{N}_{0}$
and $am=b-1-\alpha$ for certain $m\in\mathbb{N}_{0}$, then the meromorphic
function $\Gamma(s)\zeta(as+b)\zeta(as+b-\alpha)x^{-s}$ has a double
simple pole $s=-m$ with residue
\begin{align*}
 & \text{Residue}\left\{ \frac{\Gamma(s)\zeta(as+b)\zeta(as+b-\alpha)}{x^{s}},s=-m\right\} \\
 & =\frac{(-x)^{m}\left(a\zeta'(1+\alpha)+\zeta(1+\alpha)(\gamma a+\psi(m+1)-\log(x))\right)}{am!}
\end{align*}
and a simple pole $s=-m-\frac{\alpha}{a}$ with residue
\[
\text{Residue}\left\{ \frac{\Gamma(s)\zeta(as+b)\zeta(as+b-\alpha)}{x^{s}},s=-m-\frac{\alpha}{a}\right\} =\frac{x^{m+\alpha/a}}{a}\Gamma\left(-\frac{am+\alpha}{a}\right)\zeta\left(1-\alpha\right)
\]
 and simple poles at all nonpositive integers other than $-m$ with
residue
\[
\text{Residue}\left\{ \frac{\Gamma(s)\zeta(as+b)\zeta(as+b-\alpha)}{x^{s}},s=-n\right\} =\frac{(-x)^{n}}{n!}\zeta(b-an)\zeta(b-\alpha-an).
\]
Thus,
\begin{align*}
 & \sum_{n=1}^{\infty}\frac{\sigma_{\alpha}(n)}{n^{b}}e^{-n^{a}x}=\frac{(-x)^{m}\left(a\zeta'(1+\alpha)+\zeta(1+\alpha)(\gamma a+\psi(m+1)-\log(x))\right)}{am!}\\
 & +\frac{x^{m+\alpha/a}}{a}\Gamma\left(-\frac{am+\alpha}{a}\right)\zeta\left(1-\alpha\right)+\sum_{\begin{array}{c}
n=0\\
n\neq m
\end{array}}^{\infty}\frac{(-x)^{n}}{n!}\zeta(b-an)\zeta(b-\alpha-an)
\end{align*}
 as $x\to0^{+}$. 

When $am=b-1$ and $\alpha=a(m-\ell)$ for certain nonnegative integers
$m,\ell\in\mathbb{N}_{0}$ with $m\neq\ell$, then the meromorphic
function $\Gamma(s)\zeta(as+b)\zeta(as+b-\alpha)x^{-s}$ has two double
poles at $s=-m$ and $s=-\ell$ with residues
\begin{align*}
 & \text{Residue}\left\{ \frac{\Gamma(s)\zeta(as+b)\zeta(as+b-\alpha)}{x^{s}},s=-m\right\} \\
 & =\frac{(-x)^{m}\left(a\zeta'(1-a(m-\ell))+\zeta(1-a(m-\ell))(\gamma a+\psi(m+1)-\log(x))\right)}{am!}
\end{align*}
and
\begin{align*}
 & \text{Residue}\left\{ \frac{\Gamma(s)\zeta(as+b)\zeta(as+b-\alpha)}{x^{s}},s=-\ell\right\} \\
 & =\frac{(-x)^{\ell}\left(a\zeta'(1+a(m-\ell))+\zeta(1+a(m-\ell))(\gamma a+\psi(\ell+1)-\log(x))\right)}{a\ell!}
\end{align*}
 respectively. It has simple poles at all other nonpositive integers
other than $-m,\,-\ell$ with residues
\[
\text{Residue}\left\{ \frac{\Gamma(s)\zeta(as+b)\zeta(as+b-\alpha)}{x^{s}},s=-n\right\} =\frac{(-x)^{n}}{n!}\zeta(1-a(n-m))\zeta(1-a(n-\ell)).
\]
Then,
\begin{align*}
 & \sum_{n=1}^{\infty}\frac{\sigma_{\alpha}(n)}{n^{b}}e^{-n^{a}x}=\sum_{\begin{array}{c}
n=0\\
n\neq m,\ell
\end{array}}^{\infty}\frac{(-x)^{n}}{n!}\zeta(1-a(n-m))\zeta(1-a(n-\ell))\\
 & +\frac{(-x)^{m}\left(a\zeta'(1-a(m-\ell))+\zeta(1-a(m-\ell))(\gamma a+\psi(m+1)-\log(x))\right)}{am!}\\
 & +\frac{(-x)^{\ell}\left(a\zeta'(1+a(m-\ell))+\zeta(1+a(m-\ell))(\gamma a+\psi(\ell+1)-\log(x))\right)}{a\ell!}
\end{align*}
 as $x\to0^{+}$.
\end{proof}
\begin{example}
When $a=2,\ \alpha=1,\ b=\frac{1}{2}$, by (\ref{eq:2.24}) we get
\begin{equation}
\sum_{n=1}^{\infty}\frac{\sigma(n)}{\sqrt{n}}e^{-n^{2}x}=\frac{\pi^{2}}{9}\Gamma\left(\frac{7}{4}\right)x^{-\frac{3}{4}}-\Gamma\left(\frac{5}{4}\right)x^{-\frac{1}{4}}+\sum_{n=0}^{\infty}\frac{(-x)^{n}}{n!}\zeta\left(\frac{1-4n}{2}\right)\zeta\left(\frac{-1-4n}{2}\right)
\end{equation}
as $x\to0^{+}$. When $a=2,\ \alpha=1,\ b=1$, then $m=0$ in (\ref{eq:2.25}).
Then 
\[
\sum_{n=1}^{\infty}\frac{\sigma(n)}{n}e^{-n^{2}x}=\frac{\pi^{5/2}}{12\sqrt{x}}+\frac{\log x}{4}-\log\sqrt{2\pi}-\frac{\gamma}{4},
\]
as $x\to0^{+}$, it implies that the difference between two sides
of the above formula is smaller than any $x^{n}$. Let $a=2,\ b=1,\ \alpha=-2$
in (\ref{eq:2.27}), then $m=0,\ \ell=1$. Then, 
\begin{align*}
 & \sum_{n=1}^{\infty}\frac{\sigma_{-2}(n)}{n}e^{-n^{2}x}=-\frac{\zeta(3)}{2}\log x+\frac{2\zeta^{'}(3)+\zeta(3)\gamma}{2}-\frac{x\log x}{24}\\
 & +\frac{24\log A+\gamma+1}{24}x+\frac{1}{4}\sum_{n=2}^{\infty}\frac{B_{2(n-1)}B_{2n}}{(n-1)n}\frac{(-x)^{n}}{n!}
\end{align*}
 as $x\to0^{+}$.
\end{example}


\begin{thebibliography}{1}
\bibitem{Andrews} G. E. Andrews, R. Askey and R. Roy, \emph{Special
Functions, }Cambridge University Press, Cambridge, 1999.

\bibitem{Apostol}T. M. Apostol, Introduction to analytic number theory,
Undergraduate Texts in Mathematics, Springer-Verlag, New York-Heidelberg,
1976.

\bibitem{Berndt}B. C. Berndt and B. Kim, Asymptotic Expansions of
Certain Partial Theta Functions, Proceedings of AMS, Volume 139, Number
11, November 2011, 3779--3788 

\bibitem{Bringmann}K. Bringmann, A. Folsom and A. Milas, Asymptotic
behavior of partial and false theta functions arising from Jacobi
forms and regularized characters, Journal of Mathematical Physics
58, 011702 (2017); ; doi: 10.1063/1.4973634.

\bibitem{DLMF} Nist DLMF, \url{http://dlmf.nist.gov/}

\bibitem{Rademacher}Hans Rademacher, \emph{Topics in Analytic Number
Theory}, Springer-Verlag, Berlin, 1973
\end{thebibliography}
\end{document}